\documentclass[oneside]{amsart}
\usepackage{amsmath,amssymb,amsfonts,amsthm}
\usepackage{color}
\usepackage{graphicx}
\usepackage{geometry} 
\usepackage{amscd}
\usepackage{enumitem}

\title{Freedom of $h(2)$-variationality and metrizability of sprays}

\author[Elgendi]{S.G.~Elgendi}

\address{Salah G.~Elgendi, Institute of Mathematics, University of Debrecen,
  Debrecen, Hungary} \email{salah.ali@fsci.bu.edu.eg}

\author[Muzsnay]{Z.~Muzsnay}

\address{Zolt\'an Muzsnay, Institute of Mathematics, University of Debrecen,
  Debrecen, Hungary} \email{muzsnay@science.unideb.hu}
\urladdr{http://math.unideb.hu/muzsnay-zoltan}



\thanks{The research was partially supported by the European Union's Seventh
  Framework Programme (FP7/2007-2013) under grant agreements no.~318202 and
  no.~317721.}

\def\blue#1{\textcolor[rgb]{0.0,0.0,1.0}{#1}}

\newcommand{\R}{{\mathbb R}}

\def\h{{\mathfrak h}}
\newcommand{\T}{{\mathcal T}}
\newcommand{\C}{{\mathcal C}}

\newcommand{\tm}{\T M}

\newcommand {\TM}{\mathcal T\hspace{-1pt}M}

\def\el{{\mathcal E}_S} 
\def\elh#1{{\mathcal {E}{_{S,#1}}}}
\def\elhp#1{{\mathcal {E}^+_{S,#1}}}

\def\hol{{\mathcal H_{S}}}
\def\holh#1{{\mathcal H_{S \hspace{-1pt},#1}}}
\def\m{m_{\scalebox{0.5}{$S$}}} 

\def\v{{\scalebox{0.8}{\ensuremath{\mathcal{V}}}}_{\scalebox{0.7}{\ensuremath{S}}}}
\def\smallv{{\scalebox{0.6}{\ensuremath{\mathcal{V}}}}_{\scalebox{0.5}{\ensuremath{S}}}}

\def\vh#1{{\scalebox{0.8}{\ensuremath{\mathcal{V}}}}_{\scalebox{0.7}{\ensuremath{S
        \hspace{-1pt},#1}}}}

\def\H{{{\mathcal D}_{\mathcal H}}}
\def\HC{{{\mathcal D}_{\mathcal H,C}}}

\def\+{\!+\!}
\def\n{\!-\!}
\def\={\!=\!}
\def\<{\!<\!}
\def\>{\!>\!}
\def\o{\!\otimes}

 \let\oldmarginpar\marginpar
\renewcommand\marginpar[1]{\oldmarginpar[\raggedleft\footnotesize #1]%
  {\blue{\raggedright \footnotesize \fbox{
      \begin{minipage}{1.0\linewidth}
        #1
      \end{minipage}
}}}}

\numberwithin{equation}{section} 
\numberwithin{figure}{section} 

\theoremstyle{plain}
\newtheorem*{theorem*}{Theorem}
\newtheorem{theorem}{Theorem}[section]
\newtheorem{lemma}[theorem]{Lemma}
\newtheorem{proposition}[theorem]{Proposition}

\theoremstyle{definition}
\newtheorem{definition}[theorem]{Definition}
\newtheorem{property}[theorem]{Property}
\theoremstyle{remark}
\newtheorem{example}{Example}
\newtheorem{remark}[theorem]{Remark}
\newtheorem*{acknowledgement*}{Acknowledgement}

\begin{document}

\maketitle

\begin{abstract}
  In this paper we are investigating variational homogeneous second order
  differential equations by considering the questions of how many different
  variational principles exist for a given spray.  We focus our attention on
  $h(2)$-variationality; that is, the regular Lagrange function is homogeneous
  of degree two in the directional argument.  Searching for geometric objects
  characterizing the degree of freedom of $h(2)$-variationality of a spray, we
  show that the holonomy distribution generated by the tangent direction to
  the parallel translations can be used to calculate it. As a working example,
  the class of isotropic sprays is considered.
\end{abstract}

\section{Introduction}

In 1960, W.~Ambrose et al.~\cite{sprays} introduced the notion of sprays to
give an intrinsic presentation of second order ordinary differential
equations. All sprays are associated with a second order system of ordinary
differential equations and conversely, a spray can be associated with a second
order system of ordinary differential equations (SODE). In the most
interesting cases the spray can be derived from a variational principle.  A
particular class of variational sprays is composed by the metrizable
sprays. These sprays or SODEs can be viewed as the geodesic equations of a
metric.  Several papers are devoted to the inverse problem of the calculus of
variations and in particular, to the metrizability problem (see for example
\cite{AT1992, D1941,  GrMz, STP2002} and \cite{Bucataru2, crampin, Krupka,
  Krupkova, Mz08, Szilasi} respectively).  In this paper we are considering a
different aspect of the problem which is motivated by the fact that there are
sprays for which
\begin{enumerate}[label=(\alph*)]
\item there is no regular Lagrange function, that is, the spray is not
  variational,
\item there is essentially a unique Lagrange function,
\item there are several different regular Lagrange functions.
\end{enumerate}
The questions of \emph{how many different Lagrange functions} can be
associated with a spray and \emph{how to determine} this number in terms of
geometric objects are very interesting because the answers can lead to a
better understanding of the structure.  In this paper we propose to
investigate the above questions by considering the Euler-Lagrange partial
differential system associated to sprays.  In Definition \ref{def:v} we
introduce the notion of \emph{variational freedom}, denoted by $\v$, which
shows how many different variational principle can be associated to the spray
or in other words, how many essentially different regular Lagrange functions
exist for a given spray.  In general a spray $S$ is non-variational and
therefore $\v=0$.  For most of the variational cases there is a unique
variational principle admitting $S$ as a solution, that is $\v=1$. It may also
happen that $\v>1$, that is, there exist $\v$ essentially different Lagrange
functions and $\v$ essentially different variational principle associated to
$S$.

It is particularly interesting when the Lagrange functions are homogeneous.
This is the case for many examples: in general reativity, Riemannian geometry,
Finslerian geometry, etc.  That motivates the problem to investigate the
freedom of $h(k)$-variationality, when the Lagrange function must be
$k$-homogeneous.  In this paper we show that in the regular case the
\emph{holonomy distribution} can be used to determine $\vh 2$.  In Section
\ref{sec:4} we give an explicit formula how $\vh 2$ can be calculated.  As a
working example we consider the class of isotropic sprays in Section
\ref{sec:5}.

\section{Preliminaries}

In this section, we give a brief introduction that will be needed
throughout. For more details, we refer to \cite{r21, GrMz}.  Let $M$ be an
$n$-dimensional manifold, $(TM,\pi_M,M)$ be its tangent bundle and $(\T
M,\pi,M)$ the subbundle of nonzero tangent vectors.  We denote by $(x^i) $
local coordinates on the base manifold $M$ and by $(x^i, y^i)$ the induced
coordinates on $TM$.  The vector $1$-form $J$ on $TM$ locally defined by $J =
\frac{\partial}{\partial y^i} \otimes dx^i$ is called the natural
almost-tangent structure of $T M$. The vertical vector field
$\C=y^i\frac{\partial}{\partial y^i}$ on $TM$ is called the canonical or
Liouville vector field.  Recall that $\C$ is the infinitesimal generator of
the one-parameter group of (positive) homoteties.  A vector $\ell$-form $L$ on
$TM$ is homogeneous of degree $r$ if $d_{\C}L=(r-1)L$.  A scalar $p$-form
$\omega$ on $TM$ is homogeneous of degree $r$ if
$\mathcal{L}_{\C}\omega=r\omega$. In particular a function $E\in
C^\infty(\TM)$ is k-homogeneous if
\begin{equation}
  \label{eq:5}
  \mathcal{L}_{\C}E =k  E.  
\end{equation}

\subsection*{Semi-sprays and sprays}

A vector field $S\in \mathfrak{X}(\T M)$ is called a semi-spray if $JS = \C$.
If in addition $[\C, S] = S$ then $S$ is called homogeneous semi-spray, or
simply a spray.  Locally, a semi-spray can be expressed as follows
\begin{equation}
  \label{eq:spray}
  S = y^i \frac{\partial}{\partial x^i} - 2G^i\frac{\partial}{\partial y^i}.
\end{equation}
where the functions $G^i=G^i(x,y)$ are the \emph{coefficients} of the
semi-spray.  If $S$ is a spray, then the \emph{coefficients} are homogeneous
functions of degree 2 in the $y=(y^1, \dots , y^n)$ variable.  A curve $c : I
\rightarrow M$ is called is called \emph{geodesic} of a semi-spray $S$ if $S
\circ c' = c''$. Locally, $c(t) = (x^i(t))$ is a geodesic of \eqref{eq:spray}
if and only if it satisfies the equation
\begin{equation}
  \label{eq:sode}
  \frac{d^2 x^i}{dt^2} +2G^i\Big(x ,\frac{dx}{dt}\Big)=0,
\end{equation}
therefore semi-sprays can be seen as the coordinate-free version of system of
second order differential equations.

\begin{definition}
  A \emph{regular Lagrange function} $E\colon TM \to \R$ is continuous, smooth
  on $\TM$, and the matrix field \vspace{-5pt}
  \begin{equation}
    \label{eq:1}
    g_{ij} = \frac{\partial^2 E}{\partial y^i\partial y^j}
  \end{equation}
  is regular on $\T M$.  The Lagrange function $E$ is called $k$-homogeneous
  if it is homogeneous of degree $k$ in the directional argument $y=(y^i)$.  A
  regular $2-$homogeneous Lagrange function $E$ is called \emph{Finsler energy
    function} if \eqref{eq:1} is positive definite.
\end{definition}

If $E$ is a regular Lagrange function, then the $2$-form $\Omega_E:=dd_JE$ is
non-degenerate and the Euler-Lagrange equation
\begin{equation}
  \label{eq:EL}
  i_S\Omega_E+d(\mathcal L_{\C}E-E)=0
\end{equation}
uniquely determines a semi-spray $S$.  This semi-spray is called the
\emph{geodesic semi-spray} of $E$.

\medskip 

Let us consider the \emph{inverse problem}: A semi-spray $S$ on a manifold $M$
is called \emph{variational} if there exists a variational principle, that is
a regular Lagrange function $E\colon TM \to \R$, such that the stationary
curves of the functional
\begin{displaymath}
  I(\gamma)=\int E(\gamma (t), \dot \gamma (t))dt
\end{displaymath}
are the geodesic curves of the semi-spray $S$.  It is particularly
interesting, when both the semi-spray and the regular Lagrange function are
homogeneous.  This is the case for many important examples: in general
relativity, Riemannian geometry, Finslerian geometry, etc.  That motivates the
following
\begin{definition}
  A given spray $S$ on a manifold $M$ is called $h(k)-$\emph{variational} ($k
  \in \mathbb N$) if there exists a $k$-homogeneous regular Lagrange function
  $E\colon TM \to \R$ whose geodesic spray coincide with $S$.  In particular
  $S$ is \emph{metrizable}, if there exists a Finsler energy function such
  that the associated geodesic spray is $S$.
\end{definition}
J.~Szenthe proved that in the analytical case if there exists a regular
associated Lagrangian for a spray $S$, then there exists a $2-$homogeneous
regular associated Lagrangian too. This result can be reformulated as
\begin{proposition}[Szenthe \cite{Szenthe}]
  \label{sec:szenthe}
  An analytical spray on an analytical manifold is variational if and only if
  it is $h(2)-$variational.
\end{proposition}

To decide whether or not a spray $S$ can be derived from a variational
principle one has to consider the Euler-Lagrange partial differential equation
which is formally the same as \eqref{eq:el} but considered this time as a
differential system on the unknown Lagrange function $E$.
\begin{definition}
  The solutions of \eqref{eq:el} are called \emph{Euler-Lagrange functions} of
  the spray $S$. The set of Euler-Lagrange functions of $S$ will be denoted by
  $\el$. The subset of $k$-homogeneous Euler-Lagrange functions will be
  denoted by $\elh k$.
\end{definition}
Let $S$ be a spray on a manifold $M$.  The \textit{Euler-Lagrange form}
associated with $E$ is
\begin{math}
  \label{eq:omega}
  \omega_E  := i_S \Omega_E + d \mathcal{L}_C E -dE.
\end{math}
Then the Euler-Lagrange PDE equation \eqref{eq:EL} can be written as
\begin{equation}
  \label{eq:el}
  \omega_E=0.
\end{equation}
Taking the homogeneity condition \eqref{eq:5} into account we get that
\begin{equation}
  \label{eq:elh}
  \elh k=
  \left\{
    E  \in C^\infty(\TM) \ | \ \omega_E\=0, \ \mathcal L_\C E \=kE \,
  \right\}.
\end{equation}
To summarize the formalism we get the
\begin{property}
  \label{prop:remark_var}
  A semi-spray $S$ is variational if there exists a regular function
  $E \in \el$.  A spray $S$ is $h(k)$-variational if there exists a regular
  function $E \in \elh k$.
\end{property}

Several works are devoted to the inverse problem of the calculus of variations
(see for example \cite{AT1992, D1941, GrMz, STP2002} and the references
therein).  In this paper we are considering a different aspect of this
problem: \emph{How many variational principles} exist for a given spray and
how this number can be determined in terms of \emph{geometric objects}
associated to the spray?  To formulate the problem in a precise way we
introduce the following

\begin{definition} 
  \label{def:v}
  Let $S$ be a semi-spray. If $S$ is variational, then its \emph{variational
    freedom} is $\v (\in \mathbb N)$ where $\v={rank}\, (\el)$.  If $S$ is
  non-variational, then we set $\v=0$.
\end{definition}
We precise here that the notation $\v={rank}\, (\el)$ means that $\el$ can be
locally generated by its $\v$ functionally independent of its elements.  In
other words, if the variational freedom of $S$ is $\v \geq 1$ then for every
$v_0\in \TM$ there exists a neighbourhood $U\subset \TM$ and
$E_1,...,E_{\smallv}\in\el$ functionally independent on $U$ such that any
$E\in \el$ can be expressed as
\begin{displaymath}
  \hphantom{\qquad \forall \ v \in U.}
  E(v)=\varphi\big(E_1(v), \dots, E_{\smallv} (v)\big), \qquad \forall \ v \in U,
\end{displaymath}
with some function $\varphi\colon \R^{\smallv}\to \R$.

For (homogeneous) spray, we can consider the following analogous
\begin{definition}
  \label{def:vh}
  For a variational spray $S$ the $h(k)$-variational freedom is $\vh k$ if
  $\vh k={rank}\, (\elh k)$.  We set $\vh k=0$ if there is no regular element
  in $\elh k$.
\end{definition}
We remark, that a 1-homogeneous Lagrange function cannot be regular, therefore
for any sprays $S$ we have $\vh 1=0$. 

It is particularly interesting the $\vh 2={rank}\, (\elh 2)$ which is showing
how many different $2-$homogeneous Lagrange functions or variational
principles exist for the given spray.  As it was already mentioned, in many
applications (general relativity, in Riemannian and Finslerian geometry, etc)
the energy function must be 2-homogeneous. This is why in this paper we are
focusing our attention on this special case.

In Section \ref{sec:3_geometric_object} we introduce the most important
geometric objects (parallel translation, holonomy distribution, holonomy
invariant functions) needed to compute $\vh 2$ and in Section \ref{sec:4} we
determine $\vh 2$ in the case when the associated parallel translation is
regular.

\section{Holonomy invariant functions}
\label{sec:3_geometric_object}

\subsection*{Connection, parallel translation} \
\\
Every semi-spray $S$ induces an Ehresmann connection (see
\cite{Szilasi-book}).  The corresponding decomposition is $T\tm = H\tm \oplus
V\tm$, where $V\tm = \mathrm{Ker}\, \pi_*$ is the vertical and $H\tm$ is the
horizontal subbundle of $T\tm$ defined through the corresponding
\emph{horizontal and vertical projectors}.  Locally, the projectors associated
to \eqref{eq:spray} can be expressed as
\begin{math}
  h\=\frac{\delta}{\delta x^i}\o dx^i
\end{math}
and
\begin{math}
  v\=\frac{\partial}{\partial y^i}\o \delta y^i
\end{math}
where
\begin{math}
  \delta y^i\=dy^i\+N^i_j dx^j, 
\end{math}
\begin{equation}
  \label{eq:h_i}
  \frac{\delta}{\delta x^i}\=\frac{\partial}{\partial x^i}\n
  N^j_i\frac{\partial}{\partial y^j},
\end{equation}
and 
\begin{math}
  N^j_i\=\frac{\partial G^j}{\partial y^i}.
\end{math}
The modules of horizontal and vertical vector fields will be denoted by
$\mathfrak X^h(TM)$ and $\mathfrak X^v(TM)$ respectively.  

The \emph{parallel translation} of a vector along curves is defined through
horizontal lifts: Let $\gamma\colon [0,1]\to M$ be a curve such that
$\gamma(0)=p$ and $\gamma(1)=q$.  Let $\gamma^h(0)=v$, $\gamma^h(1)=w$ and
$\gamma^h$ be a horizontal lift of the curve $\gamma$, that is $\pi \circ
\gamma^h = \gamma$ and $\dot\gamma^h(t)\in H_{\gamma^h(t)}$. The parallel
translation $\tau\colon T_pM\to T_qM$ along $\gamma$ is defined as follows:
$\tau(v)=w$.

The curvature tensor $R = - \frac{1}{2}[h,h]$ of the nonlinear connection
satisfies
\begin{equation}
  \label{eq:R}
  R(X,Y)= -v[hX, hY],
\end{equation}
and characterizes the integrability of the horizontal distribution: $H\TM$ is
integrable if and only if the curvature is identically zero.  In a local
coordinate system, the curvature is given by $R\=R^i_{jk}dx^j\o dx^k
\o\frac{\partial}{\partial y^i}$ where
\begin{math}
  R^i_{jk} = \frac{\delta N^i_j}{\delta x^k} - \frac{\delta N^i_k}{\delta x^j}.
\end{math}

\medskip

\subsection*{Holonomy distribution}

\begin{definition}[\cite{Mz08}]
  The \emph{holonomy distribution} $\H$ of a spray $S$ is the distribution on
  $TM$ generated by the horizontal vector fields and their successive
  Lie-brackets, that is
  \begin{equation}
    \label{eq:14}
    \H:= \Bigl\langle \mathfrak X^h(TM)  \Bigr \rangle_{Lie}  \!
    =\Big\{[X_1,[\dots [X_{m-1},X_m]...]] \ \big| \ X_i 
    \in \mathfrak X^h(TM) \Big\}
  \end{equation}
\end{definition}
\begin{remark}
  \label{rem:H_H}
  The holonomy distribution $\H$ is the smallest involutive distribution
  containing the horizontal distribution. Using the horizontal and vertical
  projectors we have
  \begin{displaymath}
    \H = h (\H) \oplus v (\H) = H\TM \oplus v (\H).
  \end{displaymath}
  From \eqref{eq:R} we can get that the image of the curvature tensor is a
  subset of the vertical part of the holonomy distribution, that is
  $\mathrm{Im} \, R \subset v(\H)$, and we have $\H\= H\TM$ if and only if
  $R\equiv 0$.
\end{remark}
\begin{remark}
  \label{thm:N_v}
  When $\H$ is a regular distribution, then it is integrable.  Using the
  definition of parallel translation via horizontal lifts it is easy to see
  that the integral manifold through $v\in TM$, $\mathcal{O}_\tau(v)$, is the
  orbit of $v$ with respect to all possible parallel translations.  By
  Frobenius integrability theorem one can find a coordinate system $(U, z)$ of
  $\TM$ in a neighborhood of $v \in \TM$ such that the components of $\mathcal
  O_\tau\cap U$ are the sets
  \begin{equation}
    \label{eq:frob_coord_1}
    \{w\!\in\! U \,|\, z^i(w)\!=\!z_0^{i}, \, \dim \mathcal O_\tau\!+\!1
    \leq i  \leq  2n\}, \qquad |z_0^i|<\epsilon.
  \end{equation}
\end{remark}
\begin{definition}
  We say that the parallel translation is \emph{regular} if the distribution $\H$
  is regular and the orbits of the parallel translation are regular in the
  sense that and for any $v \in \TM$ there is a neighbourhood $U\subset \TM$
  such that any orbits $\mathcal{O}_\tau$ have at most one connected component
  in $U$.
\end{definition}
If the parallel translation is regular, then there exists a coordinate system
$(U, z)$ of $\TM$ in a neighborhood of any $v \in \TM$ such that in
\eqref{eq:frob_coord_1} different $z^i$ coordinates
($\dim \mathcal O_\tau\!+\!1 \leq i \leq 2n)$ correspond to different orbits
of the parallel translation.

\medskip

\subsection*{Holonomy invariant functions}

\begin{definition}
  Let $S$ be a spray. A function $E \in C^\infty(TM)$ is called \emph{holonomy
    invariant}, if it is invariant with respect to parallel translation, that
  is, for any $v\in TM$ and for any parallel translation $\tau$ we have
  $E(\tau(v))=E(v)$. The set of holonomy invariant functions will be denoted
  by $\hol$.
\end{definition}
In the case when the parallel translation is regular, the tangent spaces of
its orbits are given by the holonomy distribution $\H$, that is
$T_v\bigl(O_\tau(v)\bigr)=\H(v)$. Consequently, $E \in C^\infty(TM)$ is a
holonomy invariant function if and only if we have $\mathcal{L}_X E=0$, $X \in
\H$ that is
\begin{equation}
  \label{eq:hol}
  \hol =
  \left\{
    E \in C^\infty(\TM) \ | \ \mathcal{L}_X E=0, \ X \in \H
  \right\}.
\end{equation}
The subset of $k$-homogeneous holonomy invariant functions will be denoted by
$\holh k$. Using \eqref{eq:5} we get
\begin{equation}
  \label{eq:holh}
  \holh{k} =
  \left\{
    E \in \hol \ | \ \mathcal{L}_\C E=kE \,
  \right\}.
\end{equation}

\bigskip

\section{Euler-Lagrange functions and $h(2)$-variational freedom of sprays}
\label{sec:4}

\bigskip

We can observe the following
\begin{property}
  \label{thm:vect_space}
  The $\el$ and $\elh k$ ($k \in \mathbb N$) are vector spaces over $\R$.
\end{property}
\begin{proof}
  Both the Euler-Lagrange equation \eqref{eq:EL} and the homogeneity equation
  \eqref{eq:5} linear parial differential equations.  Therefore linear
  combination of their solutions with constant coefficients are also
  solutions.
\end{proof}
In particular, Property \ref{thm:vect_space} states that linear combination of
$2$-homogeneous Euler-Lagrange functions of $S$ are also $2$-homogeneous
Euler-Lagrange functions of $S$.  We can consider this combination as a
\emph{trivial combination} of Euler-Lagrange functions.  As the next
proposition shows, a much wider combination of homogeneous Euler-Lagrange
functions can produce new homogeneous Euler-Lagrange functions:

\begin{proposition}
  \label{thm:func_combination}
  A $1$-homogeneous functional combination of $2$-homogeneous Euler-Lagrange
  functions of a spray $S$ is also a $2$-homogeneous Euler-Lagrange functions
  of $S$.
\end{proposition}
To prove the proposition we will use  the following  
\begin{lemma} 
  \label{thm:elh_holh}
  A 2-homogeneous Lagrangian is an Euler-Lagrange function of a spray $S$ if
  and only if it is a holonomy invariant function. Using the notation
  \eqref{eq:elh} and \eqref{eq:holh} we have
  \begin{equation}
    \label{eq:2}
    \elh 2 = \holh 2.
  \end{equation}
\end{lemma}

\begin{proof}
  Let $\h \colon TTM \to \H$ be an arbitrary projection on $\H$. In
  \cite[p.~86, Theorem 1.]{Mz08} it was proven that a $2$-homogeneous Lagrange
  functions $E \colon \TM \to \R$ is a solution of the Euler-Lagrange PDE if
  and only if it satisfies the equation
  \begin{equation}
    \label{eq:dh_E}
    d_\h E=0,
  \end{equation}
  where the $d_\h$ operator is defined by the formula
  \begin{math}
    d_\h E (X) = \h X (E)=\mathcal L_{\h X}E.
  \end{math}
  Consequently \eqref{eq:dh_E} is satisfied if and only if $E$ is a holonomy
  invariant function.
\end{proof}

\bigskip

\begin{proof}[Proof of Proposition \ref{thm:func_combination}]
  Let $\varphi=\varphi(z_1,\dots, z_r)$ be a smooth 1-homogeneous function and
  consider the functional combination 
  \begin{equation}
    \label{eq:3}
    E:=\varphi\big(E_1, \dots , E_r\big).
  \end{equation}
  of $E_1, \dots , E_r \in \elh 2$, that is, $2$-homogeneous Euler-Lagrange
  functions of a spray $S$.  To prove the theorem we have to show that $E$ is
  also a $2-$homogeneous Euler-Lagrange function of $S$. Then, because of the
  $1-$homogeneity of $\varphi$ and the $2-$homogeneity of $E_i$,
  $i=1, \dots,r$, we have
  \begin{displaymath}
    E(x,\lambda y) = \varphi(E_1(x,\lambda y), ..., E_r(x,\lambda y))=
    \varphi(\lambda^2 E_1(x, y), ..., \lambda^2 E_r(x, y)) =
    \lambda^2 E(x, y),
  \end{displaymath}
  hence $E$ is $2-$homogeneous.  Moreover, using \eqref{eq:2} we have
  $E_i\in \holh 2$ and from \eqref{eq:hol} we get $\mathcal L_X E_i=0$ for any
  vector field $X\in \H$ in the holonomy distribution. Consequently, for
  $X\in \H$ we have
  \begin{displaymath} 
    \mathcal L_X E =  \frac{\partial \varphi}{\partial z^1} \! \cdot \! 
    \mathcal L_X E_1 + \dots + \frac{\partial \varphi}{\partial z^r} 
    \! \cdot \! \mathcal L_X E_r=0,\qquad 
  \end{displaymath}
  which shows that $E\in \holh 2$ and from \eqref{eq:2} we get $E \in \elh 2$.
\end{proof}
Proposition \ref{thm:func_combination} shows that functional combination of
Euler-Lagrange functions for the spray can generate new Euler-Lagrange
functions, hence new variational principles for a spray.  The variational
freedom introduced in Definition \ref{def:v} tells us how many essentially
different variational principles exist for a given spray.  The following
Theorem can be used to determine the $h(2)$-variational freedom in terms of
geometric quantities associated to the spray.

\medskip

\begin{theorem}
  \label{thm:compute_m}
  Let $S$ be a metrizable spray such that the parallel translation with
  respect to the associated connection is regular. Then
  \begin{equation}
    \label{eq:thm_vh_2}
    \vh 2  =  \mathrm{codim} \, \H.
  \end{equation}
\end{theorem}

\medskip

To prove the above theorem we need the following lemmas:

\begin{lemma}
  \label{lemma:theta}
  Let $S$ be a spray and $E_o\in \elh 2$ nonzero on $\TM$. Then $E$ is a
  $2-$homogeneous Euler-Lagrange function of $S$ if and only if
  $\theta:=E/E_o$ is a $0-$homogeneous holonomy invariant function:
  \begin{displaymath}
    E \in \elh    2 \quad \Longleftrightarrow  \quad \theta=E/E_o\in \holh 0
  \end{displaymath}

\end{lemma}
\begin{proof}
  Using Lemma \ref{thm:elh_holh} we obtain that both $E$ and $E_o$ are
  2-homogeneous holonomy invariant functions. Thus, $\theta:=E/E_o$ is a
  0-homogeneous holonomy invariant function, that is, $\theta\in \holh 0$.
  Conversely, assume that $\theta=E/E_o\in \holh 0$. Then $E=\theta E_o$ is
  $2-$homogeneous holonomy invariant function. By Lemma \ref{thm:elh_holh},
  $E$ is an Euler-Lagrange function of the spray $S$.
\end{proof}

Let us consider the smallest involutive distribution contained $\H$ and the
Liouville vector field $\C$:
\begin{displaymath}
  \HC:=\big\langle \H, \C\big\rangle_{Lie}.
\end{displaymath}
We have the following

\begin{lemma}
  \label{lemma:H_C}
  If $S$ is a spray, then $\HC$ is linearly generated by $\H$ and $\C$, that
  is, 
  \begin{equation}
    \label{eq:HC}
    \HC=Span\{\H, \C\}.
  \end{equation}
\end{lemma}

\begin{proof}
  If $\C\in \H$ then $\HC=\H$ and \eqref{eq:HC} is true.  Let us consider the
  case when $\C \not \in \H$.  Take $X,Y\in \HC$.  Using the decomposition
  $X\!=\!X_{\H}\!+\!X_{\C}$ and $Y\!=\!Y_{\H}\!+\! Y_{\C}$ corresponding to
  the directions $\H$ and $\C$ we get
  \begin{equation}
   \label{eq:13}
   [X, Y]= [X_{\H}, Y_{\H}]+ [X_{\C}, Y_{\C}] + [X_{\C}, Y_{\H}]+ [X_{\H}, Y_{\C}].
  \end{equation}
  We have $[X_{\C}, Y_{\C}]\in Span\{\C\}$ and because of the involutivity of
  $\H$ we have also $[X_{\H}, Y_{\H}]\in \H$.  Let us consider a local basis
  \begin{math}
    \mathcal B=\big \{\frac{\delta}{\delta x^1}, \dots , \frac{\delta}{\delta
      x^n}\big\}
  \end{math}
  of the horizontal space $H\TM$ introduced in \eqref{eq:h_i}. Then the
  holonomy distribution $\H$ can be generated locally by the elements of
  $\mathcal B$ and by their successive Lie brackets.  Since the spray
  coefficients $G^i(x,y)$ introduced in \eqref{eq:spray} are 2-homogeneous in
  the $y$-variables, we have $[\C,\frac{\delta}{\delta x^i}]=0$. By the Jacobi
  identity, this is also true for the successive brackets of the
  $\frac{\delta}{\delta x^i}$'s.  Now, $Y_{\H} \in \H$ can be written as a
  linear combination of the elements $Y_{\H}=g^\alpha Y_\alpha$, where
  $Y_\alpha\in \H$ can be obtained by successive brackets of the
  $\frac{\delta}{\delta x^i}$'s, and therefore $[\C, Y_\alpha]=0$.  Hence, for
  the $\C$-directional component of $X$ we have $X_{\C}=X^c\C$ with $X^c\in
  C^\infty(\TM)$ and
  \begin{displaymath}
    [X_{\C},Y_{\H}]=[X^c\C, g^\alpha Y_\alpha]= (X_{\C} g^\alpha) Y_\alpha -(Y_{\H}X^c)\C
    + X^cg^\alpha[\C, Y_\alpha] = (X_{\C} g^\alpha)Y_\alpha -(Y_{\H}X^c)\C
  \end{displaymath}
  which is an element of $Span\{\H,\C\}$. The same argument is valid for the
  fourth term in (\ref{eq:13}).
\end{proof}

\bigskip

\begin{lemma}
  \label{thm:H+C}
  If the spray $S$ is metrizable then $\C$ is transverse to $\H$ on $\TM$,
  that is
  \begin{equation}
    \label{eq:H+C}
    Span\{\H, \C\}=\H \oplus Span\{\C\}.
  \end{equation}
\end{lemma}
\begin{proof}
  If $S$ is metrizable, then there exists a Finsler energy function $E_o\in
  \elh 2$ of $S$.  Because of Proposition \ref{thm:elh_holh} we have $E_o\in
  \holh 2$. On the other hand, by using the homogeneity property of $E_o$ we
  have
  \begin{math}
    \mathcal L_{\C_v}E=2E(v) > 0
  \end{math}
  at any point $v\in \TM$.  But the derivatives of $E_o$ with respect to the
  elements of $\H$ is zero. Therefore we obtain that $\C\not \in \H$ at $v\in
  \TM$. Using Lemma \ref{lemma:H_C} we get \eqref{eq:H+C}.
\end{proof}

\bigskip

\noindent
\emph{Proof of Theorem \ref{thm:compute_m}}. %
Let us denote by $\kappa (\in \mathbb N)$ the dimension of $\H$.  We will show in
the proof that in a neighbourhood of any $v\in \TM$ one can find exactly
$\mathrm{Codim}\, \H=2n-\kappa$ locally functionally independent elements in
$\elh 2$.

As the spray $S$ is metrizable, therefore there exists a Finsler energy
function $E_o \in \elh 2$ associated to $S$.  From \eqref{eq:HC} and
\eqref{eq:H+C}, we have $\HC=\H \oplus \C$ end therefore $\dim \HC =
\kappa+1$.  Both $\H$ and $\HC$ are involutive smooth distributions on
$\TM$. By Frobenius integrability theorem one can find a coordinate system
$(U, z)$ of $\TM$ in a neighborhood of $v_0 \in \TM$, such that $z^i(v)=1$,
$z(U)=]1\!-\!\epsilon,1\!+\!\epsilon[^{2n}$ and for all $z_0^{\kappa+1},...,
z_0^{2n}$ with $|1\!-\!z_0^i|<\epsilon$, the sets
\begin{displaymath} 
  \mathcal O_\tau\!=\!\{w\!\in\! U \,|\, z^i(w)\!=\!z_0^{i}, \, \kappa\!+\!1
  \leq i   \leq  2n\},
  \qquad
  \mathcal N \!=\!\{w\in U \,|\, z^i(w)\!=\!z_0^{i}, \, \kappa\!+\!2 
  \leq  i  \leq 2n\}
\end{displaymath}
are integral manifolds of the distributions $\H$ respectively $\HC$ over
$U$. Moreover, by the regularity of the parallel translation the coordinate
neighbourhood $U$ can be choosen such a way that for any $v\in U$ the orbit
$\mathcal O_\tau(v)$ of $v$ under the parallel translations has only one
component in $U$. In this case different $z^i$ coordinates ($\kappa+1 \leq i
\leq 2n$) correspond to different orbits, hence these coordinates parametrise
the orbits of the parallel translations on $U$.  Let
\begin{equation}
  \label{eq:12}
  \H  = Span\left\{ \frac{\partial}{\partial z^1},\dots, 
    \frac{\partial}{\partial  z^{\kappa}} \right\}, 
  \qquad 
  \HC= Span \left\{\frac{\partial}{\partial z^1},\dots, 
    \frac{\partial}{\partial z^{\kappa}},
    \frac{\partial}{\partial z^{\kappa\!+\!1}}\right\}
\end{equation}
where 
\begin{math}
  Span \left\{ \frac{\partial}{\partial z^{\kappa\!+\!1}} \right\} = Span
  \left\{ \C \right\},
\end{math}
that is,
\begin{math}
  \frac{\partial}{\partial z^{\kappa\!+\!1}}= \lambda \C,
\end{math}
with $\lambda (v_0) \neq 0$.  Hence, from \eqref{eq:5} we get
\begin{equation}
  \label{eq:11}
  \frac{\partial E_o}{\partial z^{\kappa+1}} (v_0) 
  = \lambda (\C E_o)(v_0) = 2 \lambda  E_o(v_0)\neq 0.
\end{equation}
Considering the set of $0-$homogeneous holonomy invariant functions, we have
\begin{equation}
  \label{eq:hol_0}
  \theta\in \holh 0 \quad \Longleftrightarrow \quad
  \left\{ \
    \begin{aligned}
      \mathcal L_X\theta &= 0, \ \forall \ X \!\in\! \H
      \\
      \mathcal L_{\C} \theta &= 0,
    \end{aligned}
    \ \right\}
  \quad \Longleftrightarrow \quad
  \mathcal L_X\theta = 0, \ \forall \ X \!\in\! \HC.
\end{equation}
From \eqref{eq:12} and from \eqref{eq:hol_0}, it follows that $\theta\in \holh
0$ on $U$ if and only if it is a function of the variables
\begin{math}
  z^{\kappa+2}, \dots , z^{2n},
\end{math}
that is
\begin{equation}
  \label{eq:15}
  \theta = \theta(z^{\kappa+2}, \dots ,z^{2n}).
\end{equation}
By using a convenient bump function $\psi^i$ in each variable $z^i$ ($\kappa+2
\leq i \leq 2n$), we obtain smooth functions $\theta_i := \psi^i \cdot z^i \in
C^\infty(\TM)$ (no summation convention is used here), such that
$\theta_i(v_0)=1$, $\frac{d\theta_i}{dz^i}(v_0)=1$ and
$\mathrm{supp}(\theta_i)\subset U$.  It is clear that
\begin{equation}
  \label{eq:theta}
  \theta_{\kappa+2},\ \dots \ , \theta_{2n}
\end{equation}
are functionally independent $0-$homogeneous holonomy invariant functions on
some neighbourhood $\widetilde{U}\subset U$ of $v_0$ and any elements of
$\holh 0$ can be expressed on $\widetilde{U}$ as their functional combination.
The functions \eqref{eq:theta} can be used to \lq\lq\emph{modify}" the
original Euler-Lagrange function $E_o$ to obtain new elements of $\elh 2$,
functionally independent on $\widetilde{U}$. 

Indeed, let $E_i:=(1+\theta_i) E_o$ for $\kappa+2 \leq i \leq 2n$, and set
$E_{\kappa+1}:=E_o$.  Since $1+\theta_i$ are $0-$homogeneous and $E_o$ is
$2$-homogeneous holonomy invariant functions we get that
\begin{equation}
  \label{eq:7}
  E_{\kappa+1}, \, E_{\kappa+2}, \, \dots \, ,   E_{2n},
\end{equation}
are $2-$homogeneous holonomy invariant functions. Then, by Lemma
\ref{lemma:theta}, the elements of \eqref{eq:7} are in $\elh 2$. Moreover, by
the construction we have
\begin{math}
  dE_i = d\big((1\!+\!\theta_i)E_o\big) = \frac{d\theta_i }{dz^i} E_o dz^i +
  (1\!+\!\theta_i) dE_o
\end{math}
(with no summation on $i$). Hence,
\begin{displaymath}
  (dE_i)_{v_0} = (dz^i)_{v_0} + \bigl(1+\theta_i(v_0)\bigr) (dE_o)_{v_0}.
\end{displaymath}
and taking \eqref{eq:11} into account we get
\begin{alignat*}{1}
  dE_{\kappa+1}\wedge dE_{\kappa+2}\wedge \dots \wedge dE_{2n}(v_0)&
  = \big(dE_o\wedge (dz^{\kappa+2} + \theta_{\kappa+2} dE_o) \wedge \dots
  \wedge (dz^{2n} + \theta_{2n} dE_o)\big)_{v_0}
  \\
  & = (dE_o\wedge dz^{\kappa+2} \wedge \dots \wedge dz^{2n})_{v_0}
  \\
  & = 2 (\lambda E_o \, dz^{\kappa+1}\wedge dz^{\kappa+2} \wedge \dots \wedge
  dz^{2n})_{v_0} \neq 0,
\end{alignat*}
that is, the functions \eqref{eq:7} are functionally independent in some
neighbourhood $\widehat{U}\subset \widetilde{U}$ of $v_0\in \TM$.

On the other hand, let us suppose that $E\in \elh 2$ is a $2-$homogeneous
Euler-Lagrange function associated to $S$.  Using Lemma \ref{lemma:theta}, we
get that $\theta=E/E_o$ is a $2-$homogeneous holonomy invariant
function. Then, $\theta$ has the form \eqref{eq:15} on $U$ and it can thus be
expressed as a functional combination
\begin{math}
  \theta = \Psi(\theta_{\kappa+2}, \dots, \theta_{2n}).
\end{math}
Since $E_o=E_{\kappa+1}$ we get 
\begin{displaymath}
  E = \Psi \!\left(\frac{E_{\kappa+2}}{E_{\kappa+1}}, \dots,
    \frac{E_{2n}}{E_{\kappa+1}} \right) \cdot E_{\kappa+1}
\end{displaymath}
showing that $E$ is locally a functional combinations of the elements
\eqref{eq:7}.  \qed

\bigskip

\subsection*{Metrizability freedom} \
\medskip

\noindent
Similar to the notion of variational freedom, one can introduce the
\emph{metrizability freedom} of a spray $S$ showing how many functionally
independent Finsler energy functions and hence how many essentially different
Finsler metrics exist for $S$. To be more precise, let $\elhp 2$ be the set of
Finler energy functions, that is the set of regular $2-$homogeneous Lagrange
function with \eqref{eq:1} positive definite. Alike to \eqref{def:v} and
\eqref{def:vh} we set

\begin{definition} 
  \label{def:m}
  Let $S$ be a spray. If $S$ is metrizable then its \emph{metrizability
    freedom} is $\m (\in \mathbb N)$ where $\m={rank}\, (\elhp 2)$. If $S$ is
  non-variational then we set $\m=0$.
\end{definition}
We have
\begin{proposition}
  \label{prop:met}
  Let $S$ be a metrizable spray such that the parallel translation with
  respect to the associated connection is regular. Then
  \begin{math}
    \m  =  \mathrm{codim} \, \H.
  \end{math}
\end{proposition}
\begin{proof}
  Using the reasoning of Theorem \ref{thm:compute_m} we can easily prove
  Proposition \ref{prop:met}. We just remark that, using the notation
  introduced in the proof of Theorem \ref{thm:compute_m}, we have
  $E_0\in \elhp 2$ and for any $i\!=\!\kappa\!+\!2, \dots , 2n$, a
  sufficiently small nonzero constant $c_i\in\R$ can be choosen for
  $E_i\!=\!(1+c_i\theta_i)E_o$ to remain positive definite. Hence with
  $E_o=E_{\kappa+1}$ we get
  \begin{math}
    \left\{ E_{\kappa+1}, \, E_{\kappa+2}, \, \dots \, , E_{2n}
    \right\}\subset \elhp 2.
  \end{math}
  Similar argument that we used in the proof of Theorem \ref{thm:compute_m}
  shows that these elements are locally functionally independent and they
  locally generate $\elhp 2$ which proves the Proposition.
\end{proof}

\bigskip

\section{Examples: Isotropic sprays}
\label{sec:5}

Let $S$ be a spray and $R$ its curvature tensor defined by \eqref{eq:R}.  The
\emph{Jacobi endomorphism} $\Phi$ of the $S$ is defined by
\begin{equation}
  \label{eq:6}
  \Phi=i_SR.
\end{equation}
The Ricci curvature, $\text{Ric}$, and the Ricci scalar, $\rho$ are given by
$\text{Ric}=(n-1)\rho=R^i_i=\text{Tr}(\Phi)$ \cite{shen-book1}.

\begin{definition} 
  A spray $S$ is said to be \emph{isotropic} if its Jacobi endomorphism has the form
  \begin{displaymath}
    \Phi=\rho J-\alpha\otimes C,
  \end{displaymath}
  where 
  $\rho\in C^\infty(\TM)$ is the Ricci scalar and $\alpha$ is a semi-basic $1$-form 
  on $\T M$.
\end{definition}
\begin{lemma}
  \label{ker-Jacob}
  For an isotropic sprays with non vanishing Ricci scalar one has
  \begin{math}
    \dim \H \geq 2n-1.
  \end{math}
\end{lemma}
\begin{proof}
  Let $X\in H\TM$ be a horizontal vector. We have
  \begin{equation}
    \label{eq:8}
    \Phi(X)\!=\!0 \quad \Longleftrightarrow \quad 
    \rho JX\!-\!\alpha(X) \C\!=\!0  \quad \Longleftrightarrow \quad
    JX \!=\! \frac{i_X\alpha}{\rho}\, J S
    \quad    \Longleftrightarrow \quad X \!=\! \frac{i_X\alpha}{\rho}S.    
  \end{equation}
   By (\ref{eq:8}),  
  $\ker \, \Phi \cap H\TM \,=\, Span\{S\}$ and, therefore, using the
  semi-basic property of $\Phi$, we get 
  \begin{math}
    \ker \, \Phi \,=\, V\TM \oplus S
  \end{math}
  and 
  \begin{math}
    \dim \ker \, \Phi = n+1.
  \end{math}
  Hence, we have $\dim (\mathrm{Im} \Phi) =2n-(n+1)=n-1$.  On the other hand,
  by \eqref{eq:6}, $\Phi(X)\!=\!(i_SR)(X)\!=\!R(S,X)$. Thus
  $\mathrm{Im}\,\Phi\subset \mathrm{Im} \, R$ and $\dim (\mathrm{Im} R) \geq
  n-1$.  By Remark \ref{rem:H_H}, the result follows.
\end{proof}

\begin{proposition}
  \label{thm:isotropic}
  Let $S$ be an isotropic spray on an $n$-dimensional manifold $M$ with
  regular parallel translation. Then we have $\vh 2 \in \{0,1,n\}$. More
  precisely we have the following possibilities:
  \begin{enumerate}
  \item [(a)] $\vh 2=0$ if and only if $R\neq 0$ and $S$ is not variational
    (in this case $R\neq 0$);
  \item [(b)] $\vh 2=1$ if and only if $R\neq 0$ and $S$ is variational;
  \item [(c)] $\vh 2=n$, that is maximal, if and only if $R=0$.
  \end{enumerate}
\end{proposition}
\begin{proof}
  Let us first consider \emph{(c)}.  We remark that if $R=0$, then the
  holonomy is trivial and $S$ is Riemann and Finsler metrizable variational:
  an arbitrary Minowski norm extended through parallel translation defines a
  Finsler norm for $S$.  Moreover, by using Remark \ref{rem:H_H} and Theorem
  \ref{thm:compute_m}, we have
  \begin{displaymath}
    \vh 2=n \ \Leftrightarrow \  \mathrm{codim} \, \H = n \ \Leftrightarrow \  
    \dim \H = n \ \Leftrightarrow \  R= 0.
  \end{displaymath}
\noindent
\emph{(a)} We have $\vh 2=0$ if and only if $S$ is not $h(2)-$variational.  In
that case we have necessarily $R\neq 0$. Using Theorem \ref{sec:szenthe} we
get that $S$ cannot be variational.

\noindent
\emph{(b)} Let $\vh 2=1$. Then $S$ is $h(2)$-variational with an essentially
unique $h(2)-$homogeneous regular Lagrange function. Then
$\mathrm{codim} \, \H <n$ and therefore $H\TM \subsetneq  \H$. Using Remark
\ref{rem:H_H} we obtain that $R\neq 0$. Conversely, if $S$ is variational and
$R\neq 0$ then by Lemma \ref{ker-Jacob}, we have $\dim \H \geq 2n-1$. On the
other hand, Lemma \ref{thm:H+C} shows that $\C\not \in \H$ and therefore
$\dim \H \leq 2n-1$. From the two inequalities, we have $\dim \H = 2n-1$ and
hence $\mathrm{codim} \, \H = 1$.
\end{proof}

\bigskip

\subsection*{Explicite examples}

\begin{example}[$\vh 2=0$, $\mathrm{codim}\,\H=0$] \ %
  \\ Let $M= \{(x^1,x^2)\in\mathbb{R}^2:\, x^2>0\}$ and $S$ be the spray
  \eqref{eq:spray} given by the coefficents
  \begin{displaymath}
    G^1:=  \varphi \, y^1+ \frac{y^1y^2}{2x^2}, \qquad
    G^2:=\varphi \, y^2-\frac{(y^1)^2}{4},
  \end{displaymath} 
  where we used the notation $\varphi:=\big({x^2(y^1)}^2+(y^2)^2\big)^{1/2}$.
  The spray $S$ is isotropic and the coefficients of the nonlinear connection
  are given by
  \begin{displaymath}
    N_1^1=\frac{y^2}{2x^2}+\varphi+\frac{x^2(y^1)^2}{\varphi},\quad
    N_1^2=-\frac{y^1}{2}+\frac{x^2y^1y^2}{\varphi},\quad
    N_2^1=\frac{y^1}{2x^2}+\frac{y^1y^2}{\varphi}, \quad
    N_2^2=\varphi+\frac{(y^2)^2}{\varphi}.
  \end{displaymath}
  The horizontal basis is $\{h_1,h_2\}$ where
  \begin{alignat*}{1} 
    h_1&=\frac{\partial}{\partial x^1}-\left(\frac{y^2}{2x^2}+ \varphi+
      \frac{x^2(y^1)^2}{\varphi}\right)\frac{\partial}{\partial
      y^1}+\left(\frac{y^1}{2}-\frac{x^2y^1y^2}{\varphi}\right)
    \frac{\partial}{\partial y^2},
    \\
    h_2&=\frac{\partial}{\partial
      x^2}-\left(\frac{y^1}{2x^2}+\frac{y^1y^2}{\varphi}\right)
    \frac{\partial}{\partial y^1}-\left(\varphi
      +\frac{(y^2)^2}{\varphi}\right)\frac{\partial}{\partial y^2}.
  \end{alignat*}
  We have 
  \begin{alignat*}{1}
    v_{1}:=[h_1,h_2]&=\frac{4(x^2)^2+1}{4(x^2)^2}
    \left(y^1\frac{\partial}{\partial y^2}-y^2\frac{\partial}{\partial
        y^1}\right)
    \\
    v_{2}:=\big[[h_1,h_2],h_1\big]& =\frac{4(x^2)^2+1}{4x^2\varphi}
    \left(y^1y^2\frac{\partial}{\partial y^1}+\big(\varphi^2+(y^2)^2\big)
      \frac{\partial}{\partial y^2}\right).
  \end{alignat*}
  Being $v_1$ and $v_2$ linearly independent we have
  $\H =Span\{h_1, h_2, v_1, v_2\}= T\TM$. Consequently, $\C \in \H$ and
  according to Lemma \ref{thm:H+C} the spray is not variational; that is
  $\vh 2=0$.
\end{example}

\medskip

\begin{example}[$\vh 2=0$, $\mathrm{codim}\,\H>0$] \ %
  \label{ex:2}
  \\
  Let $M= \{(x^1,x^2)\in\mathbb{R}^2:\, x^2>0\}$ and $S$ the spray
  \eqref{eq:spray} given by the coefficients
  \begin{math}
    G^1=\frac{(y^1)^2}{2x^2},
  \end{math}
  $G^2=0$. The non zero coefficient of the non linear connection is
  \begin{math}
    N_1^1=\frac{y^1}{x^2}.
  \end{math}
  The horizontal basis $\{h_1,h_2\}$ and their commutator are
  \begin{displaymath}
    h_1=\frac{\partial}{\partial x^1}-\frac{y^1}{x^2}\frac{\partial}{\partial
      y^1},  \qquad %
    h_2=\frac{\partial}{\partial x^2},
    \qquad %
    v:=[h_1,h_2]=-\frac{y^1}{(x^2)^2}\frac{\partial}{\partial y^1}.
  \end{displaymath}
  One has
  \begin{math}
    \H = Span \{h_1, h_2, v\},
  \end{math}
  $\dim \H \!=\!3$ and $\mathrm{codim}\,\H \!=\!1$. For any holonomy invariant
  2-homogeneous function $E\in \holh 2$ we have
  \begin{math}
    {\mathcal L}_{h_1}E= {\mathcal L}_{h_2}E= {\mathcal L}_{v}E=0.
  \end{math}
  From the last equation we get $\frac{\partial E}{\partial y^1}=0$ and
  threfore $E$ cannot be a regular Lagrange function. From Lemma
  \ref{thm:elh_holh}, it follows that $S$ has no regular Euler-Lagrange
  function and, therefore, it cannot be variational.
\end{example}

\medskip

\begin{example}[$\vh 2=1$] \
  \\
  Let us consider on the standard unit ball $\mathbb B^n \subset \mathbb R^n$
  and the spray \eqref{eq:spray} where
  \begin{math}
    \displaystyle G^i=-\frac{\mu\langle x,y \rangle}{1+\mu \, |x|^2}y^i
  \end{math}
  with $\mu\in \mathbb R\setminus \{0\}$.  The curvature of the spray is non
  zero and isotropic. The spray is metrisable and hence variational: it is the
  geodesic spray of the Riemannian energy function
  \begin{equation} 
    \label{eq:16}
    E_\mu=\frac{1}{2}\frac{\mu(|x|^2|y|^2-\langle
      x,y\rangle^2)+|y|^2}{(1+\mu|x|^2)^2}
  \end{equation}
  of constant flag curvature $\mu\neq 0$.  From Proposition
  \ref{thm:isotropic}, we have $\vh 2=1$. Hence \eqref{eq:16} is the
  essentially unique energy function corresponding to the given spray.
\end{example}

\medskip

\begin{example}
  [$\vh 2$ is maximal] \
  \\
  One can consider the trivial example where $M=\R^n$ and the spray
  \eqref{eq:spray} where $G^i=0$. In this case the parallel translation is
  regular and the holonomy group is trivial. Hence we have $\vh 2=n$.
  \\
  We prefer to give also another, not so obvious example: Let
  $\mathbb B^n \subset \mathbb{R}^n$ be the standard unit ball and $S$ the
  spray with
  \begin{equation}
    \label{eq:10}
    G^i=-\frac{\langle a,y \rangle}{1+\langle a,x \rangle}y^i,
  \end{equation}
  where $a\in \mathbb{R}^n$ is a constant vector with $|a|<1$.  Since $R=0$,
  then $\H=H\TM$, the horizontal distribution. Hence, by Theorem
  \ref{thm:compute_m}, the metric freedom is maximal.  We remark that
  S.S.~Chern and Z.~Shen investigated in \cite{Shen-book} the family of
  Riemannian metrics associated to the norms
  \begin{equation}
    \label{fa}
    F_a=\frac{\sqrt{1-|a|^2}}{(1+\langle a,x\rangle)^2}\sqrt{|y|^2
      -\frac{2\langle a,y \rangle\langle x,y \rangle }{1+\langle a,x\rangle}
      -\frac{(1-|x|^2)\langle a,y \rangle^2 }{1+\langle a,x\rangle}}.
  \end{equation}
  The geodesic equation of \eqref{fa} is \eqref{eq:10}, but one can find other
  generating Finsler metrics too. Indeed, putting
  \begin{math}
      \displaystyle z^i= ((1+\langle a,x\rangle)y^i - \langle a,y \rangle
      x^i)/(1+\langle a,x\rangle)^2
  \end{math}
  and considering a 1-homogeneous function $\phi \colon \R^n\to \R$ we get
  \begin{equation}
    \label{general-f}
    F_\phi(x,y)=\phi \bigl(z^1(x,y),\dots, z^n(x,y)
    \bigr )
  \end{equation}
  such that $E_\phi=\frac{1}{2}F_\phi^2$ is a (not necessarily regular)
  element of $\elh 2$.  Therefore, if $F_\phi$ satisfies the regularity
  condition \eqref{eq:1}, then it is a projectively flat Finsler mertic of
  zero flag curvature with geodesic spray given by \eqref{eq:10}.  The family
  (\ref{fa}) can be considered as a special case of (\ref{general-f}) by
  choosing $\phi (z)= \bigr(\langle z,z\rangle-\langle
  a,z\rangle^2\bigl)^{1/2}$.
\end{example}

\bigskip

\section{Open problems}

In this paper we considered sprays and investigated how many essentially
different $2-$homogeneous regular Lagrange functions and $h(2)$-variational
principles exist for a given spray. We obtained the formula
\eqref{eq:thm_vh_2} in the metrizable case.  The first problem would be 

\medskip

\noindent
\textbf{Problem 1.}  Determine the $h(2)-$variational freedom without the
metrizability assumption.

\medskip

For different degrees of homogeneity, we can also consider the following:

\smallskip

\noindent
\textbf{Problem 2.}  Determine how many essentially different $k$--homogeneous
regular Lagrange functions and variational principles exist for a given spray.

\medskip 

The most interesting challenge might be the general case:

\medskip

\noindent
\textbf{Problem 3.}  Determine how many essentially different (not necessary
homogeneous) variational principles exist for a given spray.

\medskip 

This last problem can be hard to solve, since in the non-homogeneous case
there is no simple correspondence between Euler-Lagrange functions and
holonomy invariant functions.

\end{document}